\documentclass{amsart}

\usepackage{amssymb,latexsym,amsmath,amsthm}
\usepackage[table]{xcolor}
\usepackage {graphicx}
\usepackage[latin1]{inputenc}
\usepackage[english]{babel}
\usepackage{fancyhdr}
\usepackage{color}
\usepackage[T1]{fontenc}
\usepackage{amsfonts}
\usepackage{amsmath}

\usepackage{latexsym}
\usepackage{dsfont}
\usepackage{pxfonts,bbm}
\usepackage{upgreek,txfonts}

\newtheorem{coro}{{Corollary}}
\newtheorem{exa}{{ Example}}

\newtheorem{teo}{Theorem}
\newtheorem{pro}{ Proposition }

\numberwithin{equation}{section}

\date{\today}

\title{Multiple Geronimus transformations}

\author{M. Derevyagin }
\address{M. Derevyagin\\
Department of Mathematics\\
KU Leuven\\
Celestijnenlaan 200B box 2400\\
BE-3001 Leuven (Heverlee)\\
Belgium}
\email{derevyagin.m@gmail.com}
\author{J. C. Garc\'ia-Ardila}
\address{J. C. Garc\'ia-Ardila\\
Departamento de Matem\'aticas\\
Universidad Carlos III de Madrid\\
Avenida de la Universidad 30\\
28911 Legan\'es\\
 Spain}
\email{jugarcia@math.uc3m.es}
\author{F. Marcell\'an}
\address{F. Marcell\'an\\
Instituto de Ciencias Matem\'aticas (ICMAT) and  Departamento de Matem\'aticas\\
 Universidad Carlos III de Madrid\\
 Avenida de la Universidad 30\\
 28911 Legan\'es\\
  Spain}
\email{pacomarc@ing.uc3m.es}

\begin{document}

\begin{abstract}
 We consider multiple Geronimus transformations and show that they lead to
discrete (non-diagonal) Sobolev type  inner products. Moreover, it is shown that every discrete Sobolev inner product
can be obtained as a multiple Geronimus transformation. A connection with Geronimus spectral transformations for matrix orthogonal polynomials is also considered.\\

\textbf{2010 Mathematics Subject Classification:}Primary 42C05; Secondary 15A23.\\

\textbf{Keywords and Phrases:} Orthogonal polynomials, Geronimus transformation, Sobolev inner products, Cholesky decomposition, Jacobi matrix, band matrix.

\end{abstract}
\maketitle

\section{Introduction}
In this paper we basically study one of the classical problems in the theory of orthogonal polynomials  that goes back to the work of
Fej\'
er \cite{Fejer} and Shohat \cite{Sho}. It can be stated as follows. Given a nontrivial probability measure $\mu$ supported on an interval of the real line, consider the corresponding sequence of monic orthogonal polynomials $\{P_n (t) \}_{n\geq0}$. Then the problem is to find out when the sequence $\{Q_n \}_{n\geq0}$ of monic polynomials

\begin{equation}\label{GenOP}
Q_n(t):=P_n(t)+A_1^{[n]}P_{n-1}(t)+\cdots +A_{N}^{[n]}P_{n-N}(t)
\end{equation}
with real numbers $A_1^{[n]},\cdots,A_{N}^{[n]}$,  $A_{N}^{[n]}\neq0,$ and $P_{-i}=0$, for $i=1,\dots, N,$ is a family of monic orthogonal polynomials with respect to some probability measure $\nu$ supported
on an interval of the real line.

This problem has not been fully understood up until now and in the present paper we give the most thorough answer by demonstrating that in the general situation such families $\{Q_n(t)\}_{n\geq0}$ could lead to Sobolev type orthogonal polynomials, which were introduced  in a general framework in the early nineties; see \cite{MR90} and \cite{AMRR92}.
Nevertheless, few years after the Shohat publication a complete answer to the particular case of the problem, when
$$Q_n(t):=P_n(t)+A_1^{[n]}P_{n-1}(t),$$
was given by Geronimus in \cite{Ge}, providing a way to generate new families of orthogonal polynomials.
Since even nowadays it is not so easy to get access to \cite{Ge} and it is only accessible in Russian, the paper by Geronimus remained
unnoticed until the work on discrete-time Toda and Volterra lattices by Spiridonov and Zhedanov \cite{Spi}, \cite{Spir}, where they called the new family of orthogonal polynomials $\{Q_n(t)\}_{n\geq0}$ the Geronimus transformation of $\{P_n(t) \}_{n\geq0}$ (see also \cite{Zhe}).
Later on, a more general framework to the Geronimus transformation and its inverse, the Christoffel transformation,
was given in \cite{BM04}, \cite{Yoo}. In this framework both the transformations are called Darboux transformations because they are related to $UL$- and $LU$-factorization of Jacobi matrices and are discrete analogs of the famous B\"acklund-Darboux transformations from the theory of integrable systems.

To get a basic idea about \cite{BM04},\cite{Yoo}, let us consider a linear functional $\sigma$ on the linear space $\mathbb{P}$ of polynomials with real coefficients. Next, we denote by $\left<\sigma, p\right>$ the image of $p \in \mathbb{P}$ by the linear functional $\sigma$. We define the moments of such a linear functional by $\sigma_n:=\left<\sigma,t^n\right>$. In addition, for  polynomials $p$ and $\phi,$  we can introduce new linear functionals  as follows (see \cite{Yoo})

$$\left<\phi\sigma(t),p(t)\right>=\left<\phi,\sigma(t) p(t)\right>  \ \ \ \ \ \ \ \ \ \text{and}\ \ \ \ \ \ \ \ \left<(t-a)^{-1}\sigma,\phi(t)\right>=\left<\sigma,\frac{\phi(t)-\phi(a)}{t-a}\right> .$$
The canonical Geronimus transformation of the linear functional $\sigma$ corresponding to the Geronimus transformation of orthogonal polynomials can be defined as the linear functional $\hat{\sigma}$ such that

\begin{equation}\label{functional}
\hat{\sigma}=(t-a)^{-1}\sigma + \hat{\sigma}_0\delta(t-a).
\end{equation}
 Notice that the constant $\hat{\sigma}_0$ is an arbitrary real number. In the particular case $\left<\sigma,f\right>= \int_I{f d\mu_0}$ where $\mu_0$ is a nontrivial probability  measure and $a=0$ in \eqref{functional}, we have

$$\left<\hat{\sigma},fg\right>=\int_I fg d\mu +\left(\hat{\sigma}_0-\int_I d\mu\right)f(0)g(0)$$
where  $td\mu=d\mu_0$. Unfortunately, this approach doesn't allow us to deal with our main problem in the full generality and we have to go on.

In order to move to the next level of understanding of the problem in question and, so, the Geronimus transformation, we need to reconsider everything in the context of symmetric bilinear forms \cite{BvB91}, \cite{Dere}, \cite{Dur}.
To this end, let us recall that a symmetric bilinear form $B$ is a mapping $B:\mathbb{P}\times\mathbb {P}\to\mathbb{R} $  which is linear in each of its arguments and satisfies

$$B(f,g)=B(g,f).$$
As a consequence, we can associate with every symmetric bilinear form a Gram matrix $ \left(B(t^i,t^j)\right)_{i,j=0}^{\infty}:=\left(u_{i,j}\right)_{i,j=0}^{\infty}.$
If a bilinear form is given by $B(f, g) = \left <\sigma, fg\right>$, then the corresponding Gram matrix is a Hankel  matrix.

 A symmetric bilinear form is said to be quasi-definite (resp. positive definite) if the leading principal submatrices of the Gram matrix are nonsingular (resp. with determinant greater than zero). In this case the symmetric bilinear form  generates a  sequence  of orthogonal polynomials. In fact, these polynomials can be  written as follows
\begin{equation*}
P_n(t)=
\begin{vmatrix}
\mu_{0,0}&\cdots&\mu_{0,n}\\
\vdots&\vdots&\vdots\\
 \mu_{n-1,0}&\cdots&\mu_{n-1,n}\\
 1&\cdots&t^n
 \end{vmatrix}.
\end{equation*}
The interest in considering symmetric bilinear forms in general comes from the circumstances that the associated Gram matrix does not have the structure of a Hankel matrix that appears when you deal with symmetric bilinear forms associated with linear functionals. Thus it allows us to consider some different kinds of orthogonality  like the discrete Sobolev one which has attracted the attention of many authors (see \cite{Dur}, \cite{Eva}
to name a few). In this framework, it is quite natural to define the Geronimus transformation as follows (see \cite{Dere}, \cite{Ge}). 
For a nontrivial probability measure $\mu_{0}$ supported on an infinite subset $I$ of the real line, let us introduce an associated symmetric bilinear form defined on the linear space of polynomials $\mathbb{P}$ as
$$(f(t),g(t))_0=\int_{I} f(t)g(t)d\mu_0(t).$$
The Geronimus transformation of $(\cdot,\cdot)_0$ is the symmetric bilinear form  $[\cdot,\cdot]_1$ given by the formula 
$$[tg(t),f(t)]_1=[g(t),tf(t)]_1=\int_I f(t)g(t)d\mu_0(t).$$
Moreover, in \cite{Dere} the definition of the Geronimus transformation was extended to the case of the polynomial $h(t)=t^2$, and
the corresponding orthogonal polynomials and band matrices were studied there. In this case the transformation is called the double Geronimus transformation and it is associated with the family
$$Q_n(t):=P_n(t)+A_1^{[n]}P_{n-1}(t)+A_{2}^{[n]}P_{n-2}(t).$$

In this paper we start with an arbitrary polynomial $h$ of degree $\deg h=N$. Then following \cite{Dere}  we
define {\it a multiple Geronimus transformation} as 
$$[h(t)g(t),f(t)]_h=[g(t),h(t)f(t)]_h=\int_I f(t)g(t)d\mu_0(t).$$
Now we are in a position to pose the following natural question: what can be said about the symmetric bilinear form $[\cdot,\cdot]_h$
and related orthogonal polynomials, which turn out to be of the form \eqref{GenOP}? This problem is also motivated by Dur\'an in \cite{Dur},  where general results are given for symmetric bilinear forms such that  the operator of multiplication by $h$ is symmetric with respect to the bilinear form, i.e. $B (hf, g)=B (f, gh)$. Note that our Proposition 2  is a specific case of the Lemma 3 given in \cite{Dur}.

In a word, the main idea of the present paper is to show that $[\cdot,\cdot]_h$ is a discrete Sobolev inner product and to explain the structure of the band matrices generated by the recurrence relations for Sobolev type orthogonal polynomials. At the same time, taking into account \cite{DurVanAss} our results can be considered as results for a special case of matrix orthogonal polynomials. Briefly speaking, in \cite{DurVanAss} it was shown that Sobolev type orthogonal polynomials are strongly related to matrix orthogonal polynomials.

The structure of the paper is as follows. In Section 2 some basic background and notations are presented. Section 3 deals with an extension of the Geronimus transformation to the case of arbitrary polynomials $h$. More precisely, we obtain the symmetric bilinear
forms such that the operator of multiplication by $h$ is a symmetric operator with an extra
condition.  In Section 4 we study the corresponding sequences of monic orthogonal polynomials. Section 5 is focused on an interpretation of the matrix of the above multiplication operator based on a Darboux transformation with parameters. Finally, in Section 6 we establish a connection between such factorizations and block Jacobi matrices associated with matrix orthogonal polynomials deduced from the Sobolev type orthogonal polynomials. Thus we get that multiple Geronimus transformations in the scalar case yield Geronimus spectral transformations for this special case of matrix orthogonal polynomials.

\section{Preliminaries}
Here we give some basic notations and facts. Let us start by considering a symmetric bilinear form
\begin{equation}\label{a}
(f,g)_0=\int_{I} f(t)g(t)d\mu_0(t),
\end{equation}
where $\mu_0$ is a nontrivial probability measure supported on an infinite subset  $I$ of the real line. In general, if we assume that $(\cdot,\cdot)_0$ is quasi-definite, then we know  that the corresponding sequence of monic orthogonal polynomials   $\{P_n(t)\}_{n\geq0}$  satisfies a three term recurrence relation

$$tP_n(t)=P_{n+1} (t)+D_nP_{n}(t)+C_nP_{n-1}(t),$$
where $D_{n}$ and $C_{n}$ are real numbers with $C_{n}\neq 0.$

Using a matrix notation, the above expression  reads

$$tP=J_{mon}P, $$
where
\begin{equation*}
J_{mon}=
\begin{bmatrix}
D_0 &1  &   &\\
C_1 &D_1& 1&\\
    &C_2&D_2&\ddots\\
        & & \ddots&\ddots
\end{bmatrix}
\end{equation*}
is a monic Jacobi matrix and $P=\left(P_0,P_1,\cdots\right)^T$.
If  we assume that  $(\cdot,\cdot)_0$ is a positive definite bilinear form then  there is a sequence of orthonormal polynomials $\{\hat{P}_n\}_{n\geq 0}$ such that

$$t\hat{P}_n (t)=\hat{C}_{n+1}\hat{P}_{n+1}(t)+\hat{D}_{n}\hat{P}_{n}(t)+\hat{C}_{n}\hat{P}_{n-1}(t), n\geq 0.$$
Notice that in this case $\hat{C}_{n} ^{2} =C_n\geq 0$.

With the above sequence of orthonormal polynomials we can associate a tridiagonal symmetric  Jacobi matrix
of the form
\begin{equation*}
\hat{J}=
\begin{bmatrix}
\hat{D}_0 &\hat{C}_{1}  &   &\\
\hat{C}_1 &\hat{D}_1& \hat{C}_2&\\
    &\hat{C}_2&\hat{D}_2&\ddots\\
        & & \ddots&\ddots
\end{bmatrix}
\end{equation*}
such that $t\hat{P}=\hat{J} \hat{P}$, where $\hat{P}=(\hat{P}_0,\hat{P}_1,\dots)^{T}$.

In what follows we will need the $n$-th reproducing kernel $K_n (t,y)$ associated with the monic orthogonal polynomial sequence $\{P_n(t)\}_{n\geq0}$ which is defined by the formula

$$K_n(t,y)=\sum_{k=0}^{n}\frac{P_k(t)P_k(y)}{\|P_k\|^2_{\mu_0}},$$
where $\|P_k\|^2_{\mu}= \int_{I}|P_{k}(t)|^2d\mu_0,$
At the same time, there is an explicit expression for $K_n(x,y)$, which is the so-called Christoffel-Darboux formula

$$K_n(x,y)=\frac{P_{n+1}(x)P_{n}(y)-P_{n}(x)P_{n+1}(y)}{(x-y)\|P_n\|^{2}_{\mu}},$$
for $x\ne y$. We will also use the following notation for the partial derivatives of $K_n(x,y)$:
$$K_n^{(i,j)}(x,y)=\frac{\partial^{i+j}(K_n(x,y))}{\partial x^i \partial y^j}.$$

\section{An extension of the Geronimus transformation to the multiple case}

Let $h(t)$ be a monic polynomial  of $\deg h=N$. Let us define a symmetric bilinear form $[\cdot,\cdot]_h$ on the linear space $\mathbb{P} $ of all polynomials with real coefficients by

\begin{equation}\label{b}
[hf,g]_h=[f,hg]_h=\int_{I} f(t)g(t)d\mu_0(t).
\end{equation}
Clearly, this definition does not determine the bilinear form $[\cdot,\cdot]_h$ uniquely.
However, the elements of the symmetric matrix

\begin{equation}\label{c}
\hat{S}=
\begin{bmatrix}
[1,1]_h     & \cdots & [1,t^{N-1}]_h \\
 \vdots     &    \vdots             &\vdots &         \\
 [t^{N-1},1]_h    &\cdots & [t^{N-1},t^{N-1}]_h
\end{bmatrix}
=\begin{bmatrix}
s_{0,0} & \cdots & s_{0,N-1} \\
 \vdots &  \vdots &    \vdots        \\
 s_{N-1,0}&\cdots & s_{N-1,N-1}
\end{bmatrix}
\end{equation}
can be chosen arbitrarily. It should be noted that the operator of multiplication by $h$ is symmetric  with respect
to the inner product $[\cdot,\cdot]_h$. If we assume that the underlying symmetric bilinear form is quasi-definite then
the corresponding sequence of monic orthogonal polynomials $\{P_n^{*}(t)\}_{n\geq0}$ satisfies the relation

$$h(t)P^{*}_n(t)=\sum_{k=n-N}^{n+N}c^{[n]}_{k}P^{*}_k(t),$$
where $c^{[n]}_{n+N}=1$, and $c^{[n]}_{n-N}>0$ for $n\geq N$. Thus, we can associate with the sequence  $\{P_n^{*}(t)\}_{n\geq0}$
a $2N+1$ band matrix of the form
\begin{equation}\label{Srr}
J_{mon}^{*}=
\begin{bmatrix}
c^{[0]}_{0}&c^{[0]}_{1}&\cdots&c^{[0]}_{N-1}&1& & & &\\
c^{[1]}_{0}&c^{[1]}_{1}&\cdots&\cdots&c_N^{[1]}&1& & &\\
\vdots&\vdots&\ddots & & &\ddots&\ddots & &\\
c^{[N]}_{0}&c^{[N]}_{1}&\cdots &\ddots & & &c^{[N]}_{2N-1}&1&\\
0&c^{[N+1]}_{1}&\cdots& &\ddots & & &c^{[N+1]}_{2N}&1&\\
& &\ddots & & & & & &\ddots&\ddots
\end{bmatrix}.
\end{equation}

Before dealing with  the properties of the symmetric bilinear form $[\cdot,\cdot]_h$, we  will choose an appropriate basis in the linear space $\mathbb{P}$.  Namely, let us consider the basis

 $$\mathfrak{B}_h=\{t^mh^k,\text{ }k\geq 0, 0\leq m \leq N-1\}.$$
This allows us to express every polynomial $f$ as 
$$f(t)=\sum_{0\leq m\leq N-1,  k\geq 0} a_{k,m}t^mh^k(t).$$
Moreover, if we fix $k$ and define the linear operator
$$ S_{k,h}(f)(t)=\sum_{m=0}^{N-1}a_{k,m}t^{m}h^{k}(t)$$
then we have  $f=\sum_{k\geq 0}S_{k,h}(f)(t)$.

Let  $\alpha_1,\dots, \alpha_p$  be the zeros of $h(t)$  and $\beta_1\cdots\beta_p$ be their corresponding multiplicities, i.e.

$$h(t)=(t-\alpha_1)^{\beta_1}(t-\alpha_2)^{\beta_2}\cdots(t-\alpha_p)^{\beta_{p}} \text{\ \ \ \ \ \ \ with\ \ \ \ \ \ } \sum_{i=1}^{p}\beta_i=N.$$
For each $\alpha_i$, the polynomial $h(t)$ can be represented in the form
$$h(t)=(t-\alpha_i)^{\beta_i}q_i(t)\ \ \ \ \ \ \text{where} \ \ \ \ q_i(\alpha_i)\neq 0.$$
According to the Leibniz product rule for derivatives, we have that

$$h^{(r)}(t)=\sum_{k=0}^{r}\binom{r}{k}\frac{\beta_i!}{(\beta_i-k)!}(t-\alpha_i)^{\beta_i-k}q^{(r-k)}_i(t).$$
Thus,

$$h^{(j)}(\alpha_i)=0\ \ \ \ \ \text{for} \ \ \ \ j=0,\cdots,\beta_i-1.$$
If  $f$ is a polynomial  then using its representation in terms of the basis defined above we get

$$f^{(j)}(\alpha_i)=\sum_{k=j}^{N-1}a_{0,k}\frac{k!}{(k-j)!}\alpha_i^{k-j}.$$
As a consequence,  for  $i=1,\dots, p,$ one has
\begin{equation}\label{form5}
\begin{bmatrix}
f(\alpha_i)\\
f^{(1)}(\alpha_i)\\
\vdots\\
\vdots\\
f^{(\beta_i-2)}(\alpha_i)\\
f^{(\beta_i-1)}(\alpha_i)
\end{bmatrix}
=
\begin{bmatrix}
1&\alpha_i&\alpha_i^2& \alpha_i^3&\cdots&\cdots&\cdots&\alpha_i^{N-1}\\
 &1&2\alpha_i&3\alpha_i^{2}      &\cdots&\cdots&\cdots&(N-1)\alpha_i^{N-2}\\
 & &2!&6\alpha_i                &\cdots&\cdots&\cdots&(N-1)(N-2)\alpha_i^{N-3}\\
 & &  &\ddots&\ddots                   &\cdots&\cdots&\vdots\\
 & &  &      &(\beta_i-1)!&\beta_i\alpha_i    &\cdots&\frac{(N-1)!}{(N-\beta_i)!}\alpha_i^{N-\beta_i}
\end{bmatrix}_{\beta_i\times N}
\begin{bmatrix}
a_{0,0}\\
a_{0,1}\\
\vdots\\
\vdots\\
a_{0,N-2}\\
a_{0,N-1}\\
\end{bmatrix}.
\end{equation}
Introducing the matrices
\[
F_i=
\begin{bmatrix}
f(\alpha_i)\\
f^{(1)}(\alpha_i)\\
\vdots\\
\vdots\\
f^{(\beta_i-2)}(\alpha_i)\\
f^{(\beta_i-1)}(\alpha_i)
\end{bmatrix},\quad
A_i=\begin{bmatrix}
1&\alpha_i&\alpha_i^2& \alpha_i^3&\cdots&\cdots&\cdots&\alpha_i^{N-1}\\
 &1&2\alpha_i&3\alpha_i^{2}      &\cdots&\cdots&\cdots&(N-1)\alpha_i^{N-2}\\
 & &2!&6\alpha_i                &\cdots&\cdots&\cdots&(N-1)(N-2)\alpha_i^{N-3}\\
 & &  &\ddots&\ddots                   &\cdots&\cdots&\vdots\\
 & &  &      &(\beta_i-1)!&\beta_i\alpha_i    &\cdots&\frac{(N-1)!}{(N-\beta_i)!}\alpha_i^{N-\beta_i}
\end{bmatrix}_{\beta_i\times N}
\]
we see that formula \eqref{form5} for $i=1,\dots,p$, i.e.
\begin{equation*}
F_i=A_i
\begin{bmatrix}
a_{0,0}\\
\vdots\\
a_{0,N-1}\\
\end{bmatrix},
\end{equation*}
can be gathered as follows:
\begin{equation}\label{part}
\begin{bmatrix}
F_1\\
\vdots\\
F_p
\end{bmatrix}
=
\begin{bmatrix}
A_1\\
\vdots\\
A_p
\end{bmatrix}
_{N\times N}
\begin{bmatrix}
a_{0,0}\\
\vdots\\
a_{0,N-1}
\end{bmatrix}
= \mathcal{A}
\begin{bmatrix}
a_{0,0}\\
\vdots\\
a_{0,N-1}\\
\end{bmatrix}
\end{equation}

By the definition, the above system of linear equations \eqref{part} has at least a solution $[a_{0,0},a_{0,1}\cdots a_{0,N-1}]^T$.  Let us assume that there is another one which we denote by $[a^{\prime}_{0,0},a^{\prime}_{0,1}\cdots a^{\prime}_{0,N-1}]^T$. Then we define the polynomials $u(t)=\sum_{m=0}^{N-1}a_{0,m}t^m$ and $v(t)=\sum_{m=0}^{N-1}a^{\prime}_{0,m}t^m$. So, in view of \eqref{part} we have that for each $i=1,\dots, p$,
$$u^{(j)}(\alpha_i)=f^{(j)}(\alpha_i)=u^{(j)}(\alpha_i) \ \ \ \ \ \ \ \text{for}\ \ \ \ j=0,\cdots ,\beta_i-1.  $$
We now define the polynomial  $c(t)=u(t)-v(t).$ Notice that  $\deg c\leq N-1$ but, on the other hand,
$$c^{(j)}(\alpha_i)=0 \ \ \ \ \ \ \ \text{for}\ \ \ \ j=0,\cdots ,\beta_i-1.$$
This implies that  $\alpha_i$ is a zero of multiplicity at least  $\beta_i$ for $c(t)$ and since this is true for every $i=1\cdots, p,$ then $\deg c\geq N$. So, necessarily $c(t)=0,$ i.e. $u(t)=v(t)$. Therefore the solution of \eqref{part} is unique and, as a consequence, $\mathcal{A}$ is a nonsingular matrix.  In particular, if the zeros of  $h(t)$ are simple then \eqref{part} takes the form

\begin{equation}
\begin{bmatrix}
f(\alpha_1)\\
\vdots\\
f(\alpha_N)
\end{bmatrix}
=
\begin{bmatrix}
1&\alpha_1&\cdots&\alpha_1^{N-1}\\
\vdots   &\vdots  & \vdots \\
1&\alpha_N&\cdots&\alpha_N^{N-1}
\end{bmatrix}
_{N\times N}
\begin{bmatrix}
a_{0,0}\\
\vdots\\
a_{0,N-1}
\end{bmatrix}.
\end{equation}
In other words, the corresponding matrix $\mathcal{A}$ is a Vandermonde matrix.

\begin{pro}\label{prop1} Let $\mu$ be  a nontrivial probability measure with finite moments of all nonnegative orders.
Consider the measure $d\mu_0=hd\mu$. Let  $f(t)=\sum a_{k,m}t^mh^k$ and $g(t)=\sum b_{k^{\prime},m^{\prime}}t^{m^{\prime}}h^{k^{\prime}}$ be polynomials.  Then $[\cdot,\cdot]_h$ can be represented as 
\begin{equation}\label{d}
[f,g]_h =\int f(t)g(t)d\mu +
\begin{bmatrix}
F^{T}_1&\cdots & F^{T}_p
\end{bmatrix}
\mathcal{A}^{-T}S\mathcal{A}^{-1}
\begin{bmatrix}
G_1\\
\vdots\\
G_p
\end{bmatrix},
\end{equation}
where $G_i=[g(\alpha_i),\cdots, g^{(\beta_i-1)}(\alpha_i)]^T$,  $F_i=[f(\alpha_i),\cdots, f^{(\beta_i-1)}(\alpha_i)]^T$  and
the symmetric matrix $S$ has the form
\begin{equation*}
S=
\begin{bmatrix}
s_{0,0}-\int_{I} d\mu & \cdots & s_{0,N-1}-\int_{I} t^{N-1}d\mu \\
 \vdots &  \vdots &    \vdots        \\
 s_{N-1,0}-\int_{I} t^{N-1} d\mu&\cdots & s_{N-1,N-1}-\int_{I} t^{2N-2}d\mu
\end{bmatrix}.
\end{equation*}
\end{pro}

\begin{proof}
To compute $[f,g]_h$ for the given polynomials $f$ and $g$ let us observe that the polynomial
\[
f(t)-\sum_{m=0}^{N-1}a_{0,m}t^m
\]
is divisible by $h$ due to the construction. Now, we have
\begin{align*}
[f,g]_h&=\left[f(t)-\sum_{m=0}^{N-1}a_{0,m}t^m,g(t)\right]_h+\left[\sum_{m=0}^{N-1}a_{0,m}t^m,g(t)\right]\\ \notag
&= \left(\frac{f(t)-\sum_{m=0}^{N-1}a_{0,m}t^m}{h},g(t)\right)_0+\left[\sum_{m=0}^{N-1}a_{0,m}t^m,g(t)-\sum_{m^{\prime}=0}^{N-1}b_{0,m^{\prime}}t^{m^{\prime}}\right]_h+\left[\sum_{m=0}^{N-1}a_{0,m}t^m,\sum_{m^{\prime}=0}^{N-1}b_{0,m^{\prime}}t^{m^{\prime}}\right]_h\\ \notag
&=\int_{I} \left(\frac{f(t)-\sum_{m=0}^{N-1}a_{0,m}t^m}{h}\right)g(t)d\mu_0 \\
&+ \int_{I}\sum_{m=0}^{N-1}a_{0,m}t^m \left (\frac{g(t)-\sum_{m^{\prime}=0}^{N-1}b_{0,m^{\prime}}t^{m^{\prime}}}{h}\right)d\mu_0+\sum_{m=0}^{N-1}\sum_{m^{\prime}=0}^{N-1}a_{0,m}b_{0,m^{\prime}}[t^m,t^{m^{\prime}}]_h\\
&=\int_{I} f(t)g(t)d\mu-\sum_{m=0}^{N-1}\sum_{m^{\prime}=0}^{N-1}a_{0,m}b_{0,m^{\prime}}\int_{I} t^{m^{\prime}+m}d\mu+\sum_{m=0}^{N-1}\sum_{m^{\prime}=0}^{N-1}a_{0,m}b_{0,m^{\prime}}s_{m,m^{\prime}},
\end{align*}
where $s_{m,m^{\prime}}=[t^m,t^{m^{\prime}}]_h$.
In matrix form the above expression reads
\begin{align*}
[f,g]_h &=\int_{I} f(t)g(t)d\mu +\\ \notag
 &\begin{bmatrix} \notag
a_{0,0}&\cdots& a_{0,N-1}
\end{bmatrix}
\begin{bmatrix}
s_{0,0}-\int_{I} d\mu & \cdots & s_{0,N-1}-\int_{I} t^{N-1}d\mu \\
 \vdots &  \vdots &    \vdots        \\
 s_{N-1,0}-\int_{I}t^{N-1} d\mu&\cdots & s_{N-1,N-1}-\int_{I} t^{2N-2}d\mu
\end{bmatrix}
\begin{bmatrix}
b_{0,0} \\
 \vdots   \\
 b_{0,N-1}
\end{bmatrix}.
\end{align*}
Next, using \eqref{part} we get \eqref{d}.
\end{proof}

If we assume that $ h(t)=t^N$, then we have the following result that appears in \cite{Dere} for $N=2$.

\begin{coro}
If $\mu$ is a nontrivial probability measure with finite moments of all nonnegative orders then
\begin{equation}\label{coro1}
[f,g]_h=\int_{I}f(t)g(t)d\mu+
\begin{pmatrix}
f(0)&\cdots&f^{(N-1)}(0)\\
\end{pmatrix}
M
\begin{pmatrix}
g(0)\\
\vdots\\
g^{(N-1)}(0)
\end{pmatrix}
\end{equation}
where $M$ is a symmetric matrix such that
\begin{equation*}
M=
\begin{bmatrix}
\frac{1}{0!}&&\\
&\ddots&\\
&&&\frac{1}{(N-1)!}
\end{bmatrix}
S
\begin{bmatrix}
\frac{1}{0!}&&\\
&\ddots&\\
&&&\frac{1}{(N-1)!}
\end{bmatrix}.
\end{equation*}
\end{coro}

Since the values $s_{i.j}$ in \eqref{c} are arbitrary, we can take them in such a way that the matrix $S$ is diagonal, i.e.
\begin{equation*}
S=
\begin{bmatrix}
\lambda_0&       &\\
         &\ddots &\\
         &       &\lambda_{N-1}
\end{bmatrix}.
\end{equation*}
In this case \eqref{coro1} reduces to

$$[f,g]_h=\int_{I}f(t)g(t)d\mu+\sum_{k=0}^{N-1}M_k f^{(k)}(0)g^{(k)}(0)\ \ \text{with} \ \ \  \ M_k=\frac{\lambda_k}{(k!)^2},$$
which is a  diagonal discrete Sobolev inner product. In other words, we see that $N$-th iterated Geronimus transformation of $(\cdot, \cdot)_0$ generates discrete Sobolev inner products.

\section{Orthogonal polynomials associated  to the multiple Ge\-ro\-nimus transformation}

Next, assuming that the bilinear form $[\cdot,\cdot]_h$ is quasi-definite,  we will represent the monic  polynomials $\{P_n^{*}(t)\}_{n\geq0}$ orthogonal with respect to $[\cdot,\cdot]_h$,  in terms of the sequence $\{P_n(t)\}_{n\geq0}$ of monic orthogonal polynomials  with respect to  $(\cdot,\cdot)_0$ . Notice that from the orthogonality of $P_n^{*} (t)$, for the elements of the basis  $\mathfrak{B}_h$ we get
$$[P_n^{*},t^{m}h^{k}]_h=[t^{m}h^{k},P_n^{*}]_h=0 \ \text{ for} \ \ \ Nk+m\leq n-1.$$
So, for $n>N$, the definition of the bilinear form yields
$$[P_n^{*},t^{m}h^{k}]_h=(P_n^{*},t^{m}h^{k-1})_0=0  \text{ for  \ } N(k-1)+m< n-N  \text{ and\ } k\geq 1,$$
which basically means that
\begin{equation}\label{rec}
P_n^{*}(t)=P_n(t)+A^{[n]}_{n-1}P_{n-1}(t)+\cdots +A^{[n]}_{n-N}P_{n-N}(t).
\end{equation}
At the same time, we also have that
$$[P_n^{*},t^m]_h=0, \text{ \ \ \ for \ \ \ } m=0,\cdots N-1,$$
which can be rewritten as
\begin{equation*}
[P_n,t^m]_h+A^{[n]}_{n-1}[P_{n-1},t^m]_h+\cdots +A^{[n]}_{n-N}[P_{n-N},t^m]_h=0.
\end{equation*}
The latter relation is equivalent to the system of linear equations
\begin{equation}\label{si}
\begin{bmatrix}
[P_{n-1},1]_h& \cdots &[P_{n-N},1]_h \\
 \vdots & &\vdots  \\
[P_{n-1},t^{N-1}]_h& \cdots &[P_{n-N},t^{N-1}]_h
\end{bmatrix}
\begin{bmatrix}
A^{[n]}_{n-1}\\
\vdots\\
A^{[n]}_{n-N}
\end{bmatrix}
=
\begin{bmatrix}
-[P_{n},1]_h\\
\vdots\\
-[P_{n},t^{N-1}]_h
\end{bmatrix}.
\end{equation}
Since  $P^*_n(t)$ is a monic polynomial of degree $n$, we know that \eqref{si} has at least one solution. If we suppose that it has two different  solutions, then there are two monic polynomials  of degree  $n$ that satisfy the orthogonality condition. But this contradicts the uniqueness of the sequence $\{P ^*_n(t) \} _ {n\geq0}.$
Moreover, the uniqueness also gives that
\begin{equation}\label{de}
d_n^{*}=
\begin{vmatrix}
[P_{n-1},1]_h& \cdots &[P_{n-N},1]_h \\
 \vdots & &\vdots  \\
[P_{n-1},t^{N-1}]_h& \cdots &[P_{n-N},t^{N-1}]_h
\end{vmatrix}\ne 0.
\end{equation}
Further, according to Cramer's rule, the  polynomials  $P_n^{*}(t)$ can be presented as 

\begin{equation*}
P_n^{*}(t)=\frac{1}{d^{*}_n}
\begin{vmatrix}
P_{n}(t)& [P_{n},1]_h&\cdots &[P_{n},t^{N-1}]_h\\
\vdots& \vdots    &\cdots&\vdots\\
P_{n-i}(t)& [P_{n-i},1]_h&\cdots &[P_{n-i},t^{N-1}]_h\\
\vdots& \vdots    &\cdots&\vdots\\
P_{n-N}(t)& [P_{n-N},1]_h&\cdots &[P_{n-N},t^{N-1}]_h
\end{vmatrix}.
\end{equation*}
Now for  $0\leq q \leq N-1$ and $S_{0,h}(P_j)(t)=\sum_{k=0}^{N-1}c_{0,k}t^{k}$, we get
\begin{align*}
[P_j,t^q]_h&=[P_j(t)-S_{0,h}(P_j)(t)+S_{0,h}(P_j)(t),t^q]_h\\
&=\left[\sum_{m\geq 0}S_{m,h}(P_j)(t)-S_{0,h}(P_j)(t),t^q\right]_h+\left[\sum_{k=0}^{N-1}c_{0,k}t^{k},t^q\right]_h\\
&=\left( \frac{\sum_{m\geq 0}S_{m,h}(P_j)(t)-S_{0,h}(P_j)(t)}{h},t^q \right)_0+\sum_{k=0}^{N-1}c_{0,k}s_{k,q}\\
&=\sum_{m\geq 1}\int_{I}\frac{S_{m,h}(P_j)(t)}{h}t^qd\mu_{0}+\sum_{k=0}^{N-1}c_{0,k}s_{k,q}.
\end{align*}
Let us stress that the above analysis was done for  $n\geq N$. However, it is clear that  for $n\leq N$ the polynomial $P_n^{*}$
has the form
\begin{equation*}
P_n^{*}(t)=P_n(t)+A^{[n]}_{n-1}P_{n-1}(t)+\cdots +A^{[n]}_{0}P_{0}(t),
\end{equation*}
where we put $P_m(t)=0$ for $m<0$. So if we use similar arguments as above we have that for  $n\leq N$
it is true that
\begin{equation*}
P_n^{*}(t)=\frac{1}{d^{*}_n}
\begin{vmatrix}
P_{n}(t)& [P_{n},1]_h&\cdots &[P_{n},t^{n-1}]_h\\
\vdots& \vdots    &\cdots&\vdots\\
P_{n-i}(t)& [P_{n-i},1]_h&\cdots &[P_{n-i},t^{n-1}]_h\\
\vdots& \vdots    &\cdots&\vdots\\
P_{0}(t)& [P_{0},1]_h&\cdots &[P_{0},t^{n-1}]_h
\end{vmatrix}
\end{equation*}
As a last remark, let us notice that if  $n<N$ then $d^{*}_n$ is the determinant of a matrix of size $n\times n$, which does depend on $n$,
while in the other cases  $d^{*}_n$ is the determinant of a matrix  of size $N\times N$. \\
Thus, we can deduce the following.

\begin{pro}
Let $(\cdot,\cdot)_0$ be a quasi-definite bilinear form and let   $\{P_n (t)\}_{n\geq0}$  be  the  corresponding sequence of monic orthogonal polynomials. Then  the symmetric bilinear form $[\cdot,\cdot]_h$ is quasi-definite if and only if  $d_n^*\neq 0$ for all $n\in \mathbb{N}$.
Moreover, in the quasi-definite case, the sequence of the monic polynomials  $\{P_n^{*}(t)\}_{n\geq0}$ orthogonal  with respect to  $[\cdot,\cdot]_h.$ admits the representation
\begin{equation}\label{form13}
P_n^{*}(t)=\frac{1}{d^{*}_n}
\begin{vmatrix}
P_{n}(t)& [P_{n},1]_h&\cdots &[P_{n},t^{N-1}]_h\\
\vdots& \vdots    &\cdots&\vdots\\
P_{n-i}(t)& [P_{n-i},1]_h&\cdots &[P_{n-i},t^{N-1}]_h\\
\vdots& \vdots    &\cdots&\vdots\\
P_{n-N}(t)& [P_{n-N},1]_h&\cdots &[P_{n-N},t^{N-1}]_h
\end{vmatrix},
\end{equation}
where  $d^{*}_n$ is defined by \eqref{de} and
$$[P_j,t^{q}]_h=\sum_{m\geq 1}\int_{I}\frac{S_{m,h}(P_j)(t)}{h}t^qd\mu_{0}+\sum_{k=0}^{N-1}c_{0,k}s_{k,q}.$$
\end{pro}

If we assume that $\mu$ is a nontrivial probability measure such that  $hd\mu=d\mu_0$, then  $[f,g]_{\mu}=\int_I fg d\mu$ is a positive  definite  bilinear form  and we can state the following corollary.
\begin{coro}
If  $\{R_n(t)\}_{n\geq0}$ is the sequence of monic polynomials orthogonal with respect to  $[\cdot,\cdot]_{\mu},$ then the sequence of polynomials $\{P^{*}_n(t)\}_{n\geq0 }$ satisfies the connection formula
$$h(t)P_n^{*}(t)=R_{n+N}(t)+B_{n+N-1}^{[n]}R_{n+N-1}(t)+\cdots + B_{n-N}^{[n]}R_{n-N}(t),$$
as well as
\begin{equation}\label{rel}
(P^{*}_{n+N}(t),R_k(t))_0=0, \ \ \ \ \text{if} \ \ \ \ k<n
\end{equation}
\end{coro}

\begin{proof}
Notice  that  $h(t)P_n(t)$ can be written as
$$h(t)P_n(t)=\sum^{n+N}_{k=0}b^{[n]}_k R_k(t)$$
where
\begin{eqnarray*}
b^{[n]}_k=\frac{[hP_n, R_k]_{\mu}}{\|R_k\|^2_{\mu}}=\frac{(P_n, R_k)_{0}}{\|R_k\|^2_{\mu}}=
\left\{\begin{array}{cc}
0,  \quad k<n,\\ & \\
\frac{(P_n, R_k)_{0}}{||R_k||^2_{\mu}},  \quad k\geq n.\\
\end{array}\right.
\end{eqnarray*}
In other words, we have that
$$h(t)P_n(t)=\sum^{n+N}_{k=n}b^{[n]}_k R_k(t).$$
Combining this with \eqref{rec} immediately yields
$$h(t)P_n^{*}(t)=R_{n+N}(t)+B_{n+N-1}^{[n]}R_{n+N-1}(t)+\cdots+B_{n-N}^{[n]}R_{n-N}(t),$$
where
$$B_{n+N-m}=\sum_{k=0}^{min\{m,N\}}b^{[n-k]}_{N+n-m}A_{n-k}^{[n]}.$$
At the same time, we have that
$$h(t)P_n^{*}(t)=\sum_{k=0}^{N+n} c^{[n]}_k R_k(t) \ \ \text{with}\ \ \ \ c^{[n]}_k=\frac{(R_k,P^{*}_n)_0}{\|R_k\|^2_{\mu}}. $$
According to \eqref{form13}, we get that $c^{[n]}_k=0$ for $0\leq k \leq n-N-1, $ and $c^{[n]}_{n-N} \neq 0$.
Finally, taking into account that the representation of $hP_n^{*}$ in terms of the sequence  $\{R_n\}_{n\geq0}$ is unique,
we conclude that \eqref{rel} holds.
\end{proof}

\begin{exa}
Let us assume that $ h(t)=t^N $, $d\mu_0=t^{\alpha+N}e^{-t}dt,$  and define $ (\cdot, \cdot)_0 $ as
$$(f,g)_{0}=\int_{0}^{\infty}f(t)g(t)t^{\alpha+N}e^{-t}dt \ \ \ \ \ \ \ \ \ \ \alpha>-1.$$
We know that the monic orthogonal polynomials associated with the above bilinear form are the Laguerre polynomials $\{L_{n} ^{\alpha+N} \}_{n\geq0}$ with parameter $\alpha+N$. Let us now take $d\mu=t^{\alpha}e^{-t}dt$. Then
\begin{equation}\label{exa1}
[f(t), g(t)]_{h}= \int_{0}^{\infty} f(t) g(t) t^{\alpha} e^{-t} dt + \sum_{k,j=0}^{N-1} M_{k,j} f^{(k)}(0) g^{(j)}(0).
\end{equation}
As a straightforward consequence, the sequence of polynomials orthogonal with respect to \eqref{exa1} can be written as
$$\tilde{L^{\alpha}_n}(t)=L_{n}^{\alpha+N} (t)+\sum_{k=1}^{N}A_{n-k}^{[n]}L_{n-k}^{\alpha+N}(t). $$
The above bilinear form with their orthogonal polynomials is very well known in the literature. Indeed, the diagonal case was introduced in \cite{koek}. Let us notice that, in particular, if $ M_{k,j}=0$  for $(k,j)\neq (0,0)$ we get the so called Laguerre-Krall orthogonal polynomials.
\end{exa}

The previous corollary shows a connection formula between the polynomials  $\{P^*_n(t)\}_{n\geq0}$ and the polynomials $\{R_n(t)\}_{n\geq0}$. We now focus on finding necessary and sufficient conditions for the existence of the sequence of polynomials  $\{P^*_n(t)\}_{n\geq0}$.
To this end, let us notice that in the case when $P_n^*$ exists it can be represented as
\begin{equation}\label{help1}
P^*_n(t)=R_n(t)+\sum_{k=0}^{n-1}\frac{[P_n^*,R_k]_{\mu}}{\Vert R_k\Vert_{\mu}^2}R_k(t).
\end{equation}
In order to get some information out of this relation, note that \eqref{d} can be rewritten as
$$[f,g]_{\mu}= [f,g]_h-\sum_{l,w}^p\sum_{i=0}^{\beta_{l}-1}\sum_{j=0}^{\beta_{w}-1}\lambda_{i,j,l,w}f^{(i)}(\alpha_l)g^{(j)}(\alpha_w).$$
Using the orthogonality $[P_n^*,R_k]_h=0$, $k=0,\dots, n-1,$ and substituting the latter formula in \eqref{help1} we arrive at
the following:
\begin{align}
P^*_n(t)&=R_n(t)+\sum_{k=0}^{n-1}\left[-\sum_{l,w=1}^p\sum_{i=0}^{\beta_{l}-1}\sum_{j=0}^{\beta_{w}-1}
\lambda_{i,j,l,w}\left(P^{*}_n\right)^{(i)}(\alpha_l)R_k^{(j)}(\alpha_w)\right]\frac{R_k(t)}{\|R_k\|_{\mu}^2}\\ \notag
&=R_n(t)-\sum_{l,w=1}^p\sum_{i=0}^{\beta_{l}-1}\sum_{j=0}^{\beta_{w}-1}\lambda_{i,j,l,w}\left(P^{*}_n\right)^{(i)}(\alpha_l)\left(\sum^{n-1}_{k=0}\frac{R_k^{(j)}(\alpha_w)R_k(t)}{\|R_k\|_{\mu}^2}\right)\\ \notag
&=R_n(t)-\sum^{p}_{l=1}\sum_{i=0}^{\beta_l-1}\left(P_n^*\right)^{(i)}(\alpha_l)D_{i,l}(t), \notag
\end{align}
where
$$D_{i,l}(t)=\sum^{p}_{w=1}\sum_{j=0}^{\beta_{w}-1}\lambda_{i,j,l,w}K_{n-1}^{(j,0)}(\alpha_w,t).$$
In particular, for  $1\leq q\leq p$  and $1\leq k \leq \beta_q-1$, we have
$$R_n^{(k)}(\alpha_q)=\left(P_n^*\right)^{(k)}(\alpha_q)+ \sum^{p}_{l=1}\sum_{i=0}^{\beta_l-1}\left(P_n^*\right)^{(i)}(\alpha_l)D_{i,l}(\alpha_q).$$

If we define the vector

\begin{eqnarray*}
v_j^{k}(q)=
\left\{\begin{array}{cc}
 [D^{(k)}_{0,j}(\alpha_j),\cdots, 1+D^{(k)}_{k,j}(\alpha_j),\cdots, D^{(k)}_{\beta_k-1,j}(\alpha_j)], \mbox{if} \quad {j=q},\\ & \\

[D^{(k)}_{0,j}(\alpha_j),\cdots, D^{(k)}_{k,j}(\alpha_j),\cdots, D^{(k)}_{\beta_k-1,j}(\alpha_j)], \mbox{if} \quad{j\neq q}\\
\end{array} \right.
\end{eqnarray*}
then for each $q=1, \dots, p,$ we have that

\begin{equation}
\mathbb{R}_q=
\begin{bmatrix}
R_n(\alpha_q)\\
\vdots\\
R_n^{(\beta_q-1)}(\alpha_q)
\end{bmatrix}
=
\begin{bmatrix}
v_1^0(q)& v_2^0(q)&\cdots &v_p^0(q)\\
\vdots  &\vdots   &        &\vdots\\
v_1^{\beta_q-1}(q)&v_2^{\beta_q-1}(q)& &v_p^{\beta_q-1}(q)
\end{bmatrix}_{\beta_q\times N}
\begin{bmatrix}
\left(P^*_n\right)^{(0)}(\alpha_1)\\
\vdots\\
\left(P^*_n\right)^{(\beta_1-1)}(\alpha_1)\\
\vdots\\
\left(P^*_n\right)^{(0)}(\alpha_p)\\
\vdots\\
\left(P^*_n\right)^{(\beta_p-1)}(\alpha_p)
\end{bmatrix}
=\mathbb{V}_q\mathbb{P}^*.
\end{equation}

Now we are in a position to state the following result.

\begin{pro}
Let   $\mu$ be  a nontrivial probability measure and $\{R_n (t)\}_{n\geq0}$ be the sequence of orthogonal polynomials with respect to the bilinear form  $[f,g]_{\mu}:=\int_I fg d\mu$. Let  $[\cdot,\cdot]_h$  be the symmetric bilinear form defined by

$$[f,g]_h=\int_I fgd\mu +\sum_{l,w}^p\sum_{i=0}^{\beta_{l}-1}\sum_{j=0}^{\beta_{w}-1}\lambda_{i,j,l,w}f^{(i)}(\alpha_l)g^{(j)}(\alpha_w)$$
with $\lambda_{i,j,l,w}=\lambda_{j,i,w,l}$. A necessary and sufficient condition for the existence of a sequence of monic polynomials  $\{P_n^*(t)\}_{n\geq0}$ orthogonal with respect to  $[\cdot,\cdot]_h$ is that the system of linear equations

\begin{equation}
\begin{bmatrix}
\mathbb{R}_1\\
\vdots\\
\mathbb{R}_p\\
\end{bmatrix}
=
\begin{bmatrix}
\mathbb{V}_1\\
\vdots\\
\mathbb{V}_p\\
\end{bmatrix}
\mathbb{P}^*
\end{equation}
has a unique solution.
\end{pro}












As  a next step, a natural question can be posed: when is the bilinear form $[\cdot, \cdot]_h$ positive definite? It is clear that if
we suppose that the matrix $S$ given by \eqref{d} is a positive semidefinite matrix and $\mu$ is a positive measure, then  $[\cdot,\cdot]_h$ is also positive definite (for some non-regular cases see \cite{D13}, \cite{DD11}).
Indeed, for any polynomial $q(t)$ we have that
$$[q,q]_h=\int_{I} q^2d\mu+v^TSv\geq 0 $$
where
\begin{equation*}
v^T=
\begin{bmatrix}
Q^{T}_1&\cdots & Q^{T}_p
\end{bmatrix}
\mathcal{A}^{-T}.
\end{equation*}

Alternatively, in order to analyze the positivity of  $[P_n^{*}, P^{*}_{n}]_h$, we need to consider two cases: when  $n=m+Nk$  and $n<N$.

\noindent {\textbf{Case 1.}}] If $n=m+Nk$ with $k\neq 0$ then

\begin{align}
[P^{*}_{n},t^{m}h^{k}]_h&=\int_{I}P^{*}_{n}t^{m}h^{k}d\mu=\int_{I}P^{*}_{n}t^{m}h^{k-1}d\mu_{0}\\
&=\frac{1}{d^*_{n}}
\begin{vmatrix}\notag
\int P_{n}t^mh^{k-1}d\mu_0 & [P_{n},1]_h&\cdots&[P_{n},t^{N-1}]_h\\
\vdots&\vdots& &\vdots\\
\int P_{n-N+1}t^mh^{k-1}d\mu_0 & [P_{n-N+1},1]_h&\cdots&[P_{n-N+1},t^{N-1}]_h\\
\int P_{n-N}t^mh^{k-1}d\mu_0 & [P_{n-N},1]_h&\cdot&[P_{n-N},t^{N-1}]_h\\
\end{vmatrix}.
\end{align}
But taking into account that $m+N(k-1)=n-N$, the above expression becomes
\begin{equation*}
=\frac{1}{d^*_{n}}
\begin{vmatrix}
0 & [P_{n},1]_h&\cdots&[P_{n},t^{N-1}]_h\\
\vdots&\vdots& &\vdots\\
0 & [P_{n-N+1},1]_h&\cdots&[P_{n-N+1},t^{N-1}]_h\\
\int P_{n-N}t^mh^{k-1}d\mu_0 & [P_{n-N},1]_h&\cdot&[P_{n-N},t^{N-1}]_h\\
\end{vmatrix}
.
\end{equation*}
Thus
$$[P_n^{*}(t),t^{m}h^{k}]_h=(-1) ^{N} \frac{d^*_{n+1}}{d^*_{n}}\int P_{n-N}t^mh^{k-1}d\mu_0= (-1)^{N}\frac{d^*_{n+1}}{d^*_{n}}\int P^2_{n-N}(t)d\mu_0.$$

\noindent{\textbf{Case 2.}} If $n<N$ then we have

\begin{equation}
[P^{*}_{n},t^{n}]_h=\frac{1}{d^*_{n}}
\begin{vmatrix}\notag
[P_n,t^n]_h & [P_{n},1]_h&\cdots&[P_{n},t^{n-1}]_h\\
\vdots&\vdots& &\vdots\\
[P_0,t^n] & [P_{0},1]_h&\cdot&[P_{0},t^{n-1}]_h\\
\end{vmatrix}
.
\end{equation}
Thus
\begin{eqnarray}
[P_n^{*}(t),t^{n}]_h=\left\{\begin{array}{cc}
 -\frac{d^*_{n+1}}{d^*_{n}}& \mbox{if}\quad n \ \ \ is \ \ odd,\\ & \\
\frac{d^*_{n+1}}{d^*_{n}},  & \mbox{if}\quad n\ \ \ is \ \  even
\end{array} \right.
\end{eqnarray}

As a summary we can state the following.

\begin{pro} Let  $(\cdot,\cdot)_{0}$ be a positive definite bilinear form. Then  $[\cdot,\cdot]_h$ is a positive definite bilinear form if and only if
$d_n^*\ne 0$ and
\begin{eqnarray}
\left\{\begin{array}{cc}
(-1)^{N}\frac{d^*_{n+1}}{d^*_{n}}>0 & \mbox{for}\quad n\geq N \ \\ & \\
\frac{d^*_{n+1}}{d^*_{n}}>0  & \mbox{for}\quad n<N, \ \ with \  n \ even \\ & \\
 \frac{d^*_{n+1}}{d^*_{n}} < 0& \mbox{for}\quad n<N \ \  with \ n \ \ odd,\\ & \\
\end{array} \right.
\end{eqnarray}

\end{pro}


\section{Matrix representation of the multiple Geronimus transformation}

Let us assume that $[\cdot,\cdot]_h$ is a positive definite bilinear form.  We define the symmetric matrix

\begin{equation*}
J^{*}=\left([h\hat{P^{*}}_n(t),\hat{P^{*}}_m(t)]_h\right)_{n,m=0}^{\infty},
\end{equation*}

where the corresponding orthonormal  polynomials $\{\hat{P^*_n}(t)\}_{n\geq0}$ are related to the monic ones in the following way:

$$\hat{P^*_n}(t)=\frac{1}{h^{*}_n}P_n^{*}(t), \ \ \ \ \ \ \ \ (h^{*}_n)^2=[P^{*}_n,P^{*}_n], \ \ \ h^{*}_n>0$$

For the classical Geronimus transformation (i.e. $h(t)=t$) there are two important facts concerning
the matrix factorizations  \cite{Dere}, \cite{Yoo}.
\begin{enumerate}
\item $J^*$  can be decomposed as $J^{*}=CC^{T}$ with $C$ a lower triangular matrix (Cholesky factorization).
\item If $P_n(0)\neq0$ for $n=0,1,2,\dots$ then  there exist $U$, an upper triangular matrix, and $L$, a lower triangular matrix,  such that
$$J_{mon}=UL\ \ \ \ \ \ \ \text{and} \ \ \ \ \ \ J^{*}_{mon}=LU, $$
where $J_{mon}$, $J_{mon}^*$ are monic Jacobi matrices associated with the corresponding monic orthogonal polynomials.
\end{enumerate}

Next, it is natural to ask if it is possible to extend these two results to the generalized Geronimus transformations analyzed in the previous sections. The answer to this question can be given by mimicking the idea of \cite{Dere}. Namely, from \eqref{rec} we know that the polynomials $P_n^{*}(t)$ can be written in terms of the monic orthogonal polynomials $P_n^(t)$, which are orthogonal with respect to $(\cdot,\cdot)_0$. From this we get

\begin{align*}
(P^*_n,P^*_m)_0&=A^{[n]}_{n}\sum_{i=0}^{N}A_{m-i}^{[m]}(P_{n},P_{m-i})_0+A^{[n]}_{n-1}\sum_{i=0}^{N}A_{m-i}^{[m]}(P_{n-1},P_{m-i})_0\cdots\\
&+A^{[n]}_{n-j}\sum_{i=0}^{N}A_{m-i}^{[m]}(P_{n-j},P_{m-i})_0+\cdots A^{[n]}_{n-N}\sum_{i=0}^{N}A_{m-i}^{[m]}(P_{n-N},P_{m-i})_0
\end{align*}

\begin{eqnarray}\label{prod}
=\left\{\begin{array}{cc}
\sum^{N}_{k=t}A_{n+t-k}^{[n+t]}A_{n+t-k}^{[n]} h^2_{n-k+t}, & \mbox{if}\quad m=n+t, \ 0\leq t\leq N,\\& \\
\sum^{N}_{k=t}A_{n-k}^{[n]}A_{n-k}^{[n-t]} h^2_{n-k}, & \mbox{if}\quad m=n-t,   \ 0\leq t\leq N,
\end{array} \right.
\end{eqnarray}
where $A_k^{[k]}=1$  and $A_m^{[k]}=0$ if $m<0$. Notice that $(P^*_n,P^*_m)_0$ is zero for $|n-m |\geq N$  and, therefore,
the matrix
\[
J^{*}=\left([h\hat{P^{*}}_n(t),\hat{P^{*}}_m(t)]_h\right)_{n,m=0}^{\infty}=
\left((\hat{P^{*}}_n,\hat{P^{*}}_m)_0\right)_{n,m=0}^{\infty}
\] is a $(2N+1)\times (2N+1)$ diagonal matrix.

\begin{pro}
Let us assume that $(\cdot,\cdot)_0$ and $[\cdot,\cdot]_h$ are  positive definite bilinear forms and  $\{P_n(t)\}_{n\geq0}$, $\{P^{*}_n(t)\}_{n\geq 0}$ are, respectively, the corresponding sequences of monic orthogonal polynomials. Then the symmetric matrix  $J^{*}$ corresponding
to $\hat{P^{*}}_n$ can be represented as
$$J^{*}=CC^{T},$$
where $C$ is a lower triangular matrix with positive diagonal entries,

\begin{equation*}
C=
\begin{bmatrix}
\frac{h_0}{h_0^{*}} & &  & & & &\\

\frac{A^{[1]}_0h_0}{h_1^{*}} & \frac{h_1}{h_1^{*}}& & & & &\\
\frac{A^{[2]}_0h_0}{h_2^{*}}&\frac{A_1^{[2]}h_1}{h_2^{*}}&\frac{h_2}{h_2^{*}} & & & &\\
\vdots&\vdots & \ddots & \ddots& & &\\
0&\frac{A_1^{[N+1]}h_1}{h_{N+1}^{*}}&\cdots&\frac{A_N^{[N+1]}h_N}{h_{N+1}^{*}}&\frac{h_{N+1}}{h_{N+1}^{*}}\\
 \vdots&\vdots & & &\ddots & \ddots
\end{bmatrix}.
\end{equation*}
\end{pro}

\begin{proof}
According to the definition of $J^{*}$ we have
\begin{equation*}
J^{*}=
\begin{bmatrix}
\frac{1}{h_1^{*}}&0&\\
&\frac{1}{h_2^{*}}&\ddots\\
& &\ddots&\ddots
\end{bmatrix}
\begin{bmatrix}
[hP^{*}_{0},P^{*}_{0}]_h&[hP^{*}_{0},P^{*}_{1}]_h & \\
[hP^{*}_{1},P^{*}_{0}]_h&[hP^{*}_{1},P^{*}_{1}]_h&\ddots\\
&\ddots&\ddots
\end{bmatrix}
\begin{bmatrix}
\frac{1}{h_1^{*}}&0&\\
 &\frac{1}{h_2^{*}}&\ddots\\
 &  &\ddots&\ddots
\end{bmatrix}.
\end{equation*}
Taking into account the  definition of $[\cdot,\cdot]_h$ we have that $[hP_n^{*},P_m^{*}]_h=(P_n^{*},P_m^{*})_0$.
From   \eqref{prod} we get that
\begin{equation*}
\left[(P_n^{*},P_m^{*})_0\right]_{n,m=0}^{\infty}=
{\footnotesize
\begin{bmatrix}
h_0^2 &  A^{[1]}_0h^2_0 & A^{[2]}_0h^2_0 &\cdots& A^{[N]}_0h^2_0 & 0    \\
A^{[1]}_0h^2_0 &\sum_{k=0}^{N}(A^{[1]}_{1-k})^2h^2_{1-k}&\sum_{k=0}^{N}A^{[3]}_{3-k}A^{[2]}_{3-k}h^2_{3-k}&\cdots&\sum_{k=N-1}^{N}A^{[N]}_{N-k}A^{[2]}_{N-k}h^2_{N-k}&  A_1^{[N+1]}h_1^2\\
A^{[2]}_0h^2_0 &\sum_{k=1}^{N}A^{[2]}_{2-k}A^{[1]}_{2-k}h^2_{2-k}&\sum_{k=0}^{N}(A^{[2]}_{2-k})^2h^2_{2-k}&\cdots &\cdots &\cdots\\
\vdots&\vdots&\ddots& \cdots&\vdots&\vdots\\
\vdots&\vdots&\vdots& \ddots&\vdots&\vdots\\
A^{[N]}_0h^2_0 &\sum_{k=N-1}^{N}A^{[N]}_{N-k}A^{[1]}_{N-k}h^2_{N-k}&\cdots &\cdots&\sum_{k=0}^{N}(A^{[N]}_{N-k})^2h^2_{N-k}&\sum_{k=1}^{N}A^{[N+1]}_{N+1-k}A^{[N]}_{N+1-k}h^2_{N+1-k}\\
0 & A_1^{[N+1]}h_1^2& \ddots &\ddots&\ddots&\ddots\\
\end{bmatrix}
}.
\end{equation*}
It is easy to see that this can be written as
\begin{equation*}
\begin{bmatrix}
h_0 & &  & & & &\\
A^{[1]}_0h_0 & h_1& & & & &\\
A^{[2]}_0h_0&A_1^{[2]}h_1&h_2 & & & &\\
\vdots&\vdots & \ddots & \ddots& & &\\
0&A_1^{[N+1]}h_1&\cdots&A_N^{[N+1]}h_N&h_{N+1}\\
 \vdots&\vdots & & &\ddots & \ddots
\end{bmatrix}
\begin{bmatrix}
h_0 &A^{[1]}_0h_0 &A^{[2]}_0h_0&\cdots &\cdots & 0 &\cdots\\
 &  h_1& A_1^{[2]}h_1  &\cdots &\cdots&A_1^{[N+1]}h_1&\cdots\\
 & & h_2&\ddots& & \vdots &\\
& &  &  \ddots&\ddots &\vdots &\\
& &  & & \ddots &A_N^{[N+1]}h_N&\cdots\\
& & & & & h_{N+1}&\cdots\\
& & & & &        &\ddots
\end{bmatrix}.
\end{equation*}
If we set
\begin{equation}
C=
\begin{bmatrix}
\frac{1}{h_1^{*}}&0&\\
 0&\frac{1}{h_2^{*}}&\ddots\\
& &\ddots&\ddots
\end{bmatrix}
\begin{bmatrix}
h_0 & &  & & & &\\
A^{[1]}_0h_0 & h_1& & & & &\\
A^{[2]}_0h_0&A_1^{[2]}h_1&h_2 & & & &\\
\vdots&\vdots & \ddots & \ddots& & &\\
0&A_1^{[N+1]}h_1&\cdots&A_N^{[N+1]}h_N&h_{N+1}\\
 \vdots&\vdots & & &\ddots & \ddots
\end{bmatrix}
\end{equation}
then we get the desired result.  Also, notice that
\begin{align*}
 (h_{n+N}^{*})^2&=[P^{*}_{n+N}(t),P^{*}_{n+N}(t)]_h=[hP^{*}_{n}(t),P^{*}_{n+N}(t)]_h\\
&=(P^{*}_{n}(t),P^{*}_{n+N}(t))_0=\sum_{k=N}^{N}A_{n+N-k}^{[n+N]}A_{n+N-k}^{[n]}h^2_{n+N-k}\\
&=A_n^{[n+N]}h_n^2.
\end{align*}
Hence the diagonal entries of  $C$  can be given in terms of the coefficients $A^{[k]}_{n}$ as 
$$\frac{h_{n+N}}{h_{n+N}^{*}}=\frac{h_{n+N}}{\sqrt{A_n^{[n+N]}}h_n}.$$
In addition, if $m<N$ then
\begin{align*}
(h_m^*)^2&=[P_m^*,P_m^{*}]_h=[\sum_{k=0}^{m}A_k^{[m]}P_k,\sum_{j=0}^{m}A_k^{[m]}P_j]_h\\
&=\sum_{k=0}^{m}\sum_{j=0}^{m}A_k^{[m]}A_k^{[m]}[P_k,P_j]_h.
\end{align*}
From the above relation we can see that  $(h_m^*)^{2}$ is a combination of the free parameters given by the matrix $\hat{S}$
(see \eqref{c}).
\end{proof}

Let $L_{mon}$ be the matrix associated  with the recurrence formula given in \eqref{rec}, that is
\begin{equation}
L_{mon}=
\begin{bmatrix}
1 & &  & & & &\\
A^{[1]}_0 & 1& & & & &\\
A^{[2]}_0&A_1^{[2]}&1 & & & &\\
\vdots&\vdots & \ddots & \ddots& & &\\
A^{[N]}_{0}&A^{[N]}_{1}&\cdots &A^{[N]}_{N-1}&1\\
0&A_1^{[N+1]}&\cdots&\cdots&A_N^{[N+1]}&1\\
 \vdots&\vdots & & & &\ddots & \ddots
\end{bmatrix}.
\end{equation}
It is clear that the relation  \eqref{rec} reads as  $P^*=L_{mon}P,$ where  $P^{*}=(P^{*}_0(t),P^{*}_1(t)\cdots)^{T}$ and $P=(P_0(t),P_1(t)\cdots)^{T}$.
 On the other hand, we have
$$[hP_n,P_m^{*}]_h=(P_n,P_m^{*})_0=0, \text{\ \ for\ \ } m=0,\dots, n-1.$$
Then we can write
\begin{equation}
h(t) P_n(t)=\sum_{i=n}^{N+n}B^{[N+n]}_{i}P_i^{*}(t), \text{\ \ \ where\ \ } B_{n}^{[N+n]}\neq 0.
\end{equation}
Thus we can associate with the above relation the matrix
\begin{equation*}
U_{mon}=
\begin{bmatrix}
B^{[N]}_{0}&B^{[N]}_{1}&\cdots &\cdots &B^{[N]}_{N-1}& 1 & & &\\
  &B^{[N+1]}_{1}&\cdots& \cdots &B^{[N+1]}_{N-1}&B^{[N+1]}_{N}&1& &\\
 & &\ddots &  & &\ddots&\ddots&\ddots&\\
 & & &B^{[n+N]}_{n} & & &B^{[n+N]}_{n+N-2} &B^{[n+N]}_{n+N-1}&1\\
 &&&&\ddots&&&\ddots&\ddots&\ddots
\end{bmatrix}.
\end{equation*}
Here $hP=U_{mon}P^{*}$ where $P$ and $P^{*}$ are the vectors defined as above. Finally, we can state the following.
\begin{pro}\label{prop6} If $h(t)=\sum_{m=0}^{N}b_mt^m,$ then
 \begin{equation}\label{J}
h\left(J_{mon}\right)=\sum_{m=0}^{N}b_mJ_{mon}^m=U_{mon}L_{mon}
\end{equation}
as well as
 \begin{equation}\label{jm}
J_{mon}^{*}=L_{mon}U_{mon},
\end{equation}
where $J_{mon}^*$ is the band matrix corresponding to the monic Sobolev orthogonal polynomials generated
by $[\cdot,\cdot]_h$ (see \eqref{Srr}).
\end{pro}
\begin{proof}
By definition, we have
\begin{equation*}
hP=U_{mon}P^{*}=U_{mon}L_{mon}P.
\end{equation*}
Next, observing that
\begin{equation*}
t^mP=J_{mon}t^{m-1}P=\cdots=J^{m}_{mon}P
\end{equation*}
we arrive at
\begin{equation*}
hP=\sum_{m=0}^{N}b_mt^mP=\sum_{m=0}^{N}b_mJ_{mon}^mP=h\left(J_{mon}\right)P.
\end{equation*}
From this relation and due to the uniqueness of coefficients in recurrence relations we obtain \eqref{J}. To prove \eqref{jm}, notice that

\begin{equation*}
hP^{*}=L_{mon}hP=L_{mon}U_{mon}P^{*},
\end{equation*}
Since we have
\begin{equation*}
hP^{*}=J^{*}_{mon}P^{*},
\end{equation*}
the relation \eqref{jm} is rather obvious.
\end{proof}

\section{Discrete Sobolev inner products as multiple Geronimus transformations}

In this section we summarize all the previous findings together with the results of \cite{DurVanAss} and present the main results of the present paper for a special class of polynomials.

Consider the discrete Sobolev inner product
\[
\langle f,g\rangle=\int f(t)g(t)d\mu(t)+\sum_{i=1}^{M}\sum_{j=0}^{M_i}\lambda_{i,j}f^{(j)}(\alpha_i)g^{(j)}(\alpha_i),
\]
where $f$, $g$ are polynomials and $\lambda_{i,j}$ are real numbers. We also suppose that the inner product $\langle \cdot,\cdot\rangle$ is symmetric, i. e.
$\langle f,g\rangle=\langle g,f\rangle$. Then the following result holds true.
\begin{teo}
The discrete Sobolev inner product $\langle \cdot,\cdot\rangle$ is a multiple Geronimus transformation of a bilinear form generated by
the measure $d\mu_0(t)=h(t)d\mu(t)$, where
\[
h(t)=\prod_{i=1}^{M}(t-\alpha_i)^{M_{i+1}},
\]
that is
\[
\langle f,g\rangle\equiv [f,g]_h.
\]
\end{teo}
\begin{proof}
This statement is a straightforward combination of Proposition \ref{prop1} and \cite[Section 3.1]{DurVanAss}.
\end{proof}
This result together with Proposition \ref{prop6} gives us an understanding of the structure of the band matrices associated with the recurrence
relations generated by Sobolev orthogonal polynomials.
\begin{teo}\label{th2}
Let us consider a discrete Sobolev inner product $\langle \cdot,\cdot\rangle$. Then the band matrix $J_{mon}^*$ generated by the recurrence relations for the corresponding orthogonal polynomials can be obtained as 
\begin{equation}\label{MatrGer}
h(J_{mon})=U_{mon}L_{mon}\mapsto J_{mon}^*=L_{mon}U_{mon},
\end{equation}
where $J_{mon}$ is the monic Jacobi matrix associated with $d\mu_0$.
\end{teo}

Let $p(t) = \sum _{j=0}^{n} \sum_{k=0}^{N-1} a_{k,j} t^{k} h^{j} (t)$ be a polynomial of degree $nN + m$, $0\leq m < N,$  where we assume $a_{k,n}= 0$ if  $k> m.$ For $0 \leq k < N-1, $ let consider the linear operator $R_{k,h}(p)(t)=  \sum_{j=0}^{n} a_{k,j} t^{j},$ i.e. it takes from $p$ the terms of the form  $a_{k,j} t^{k} h^{j}(t)$ and then removes the common factor $t^{k}$ and changes $h(t)$ to $t.$  Notice that in such a way $p(t)= \sum_{k=0}^{N-1} t^{k} R_{k,h} (p) (h(t))$ (see \cite{DurVanAss}).\\

Using the previous notation, Theorem \ref{th2} can be seen as a result for matrix orthogonal polynomials due to \cite{DurVanAss}. Indeed,
the matrix $h(J_{mon})$ generates matrix polynomials
\[
P_n(t)= \begin{pmatrix}
      R_{0,h}(p_{nN})(t)&\dots&R_{N-1,h}(p_{nN})(t) \\
      R_{1,h}(p_{nN+1})(t)&\dots&R_{N-1,h}(p_{nN+1})(t) \\
      \vdots & & \vdots\\
      R_{N-1,h}(p_{nN+N-1})(t)&\dots&R_{N-1,h}(p_{nN+N-1})(t) \\
   \end{pmatrix}
\]
orthogonal with respect to the measure $dM_0(h^{-1})$, where
\[
d{M}_0(t) = \begin{pmatrix}
      d\mu_0(t) &td\mu_0(t)&\dots&t^{N-1}d\mu_0(t) \\
      td\mu_0(t) &t^2d\mu_0(t)&\dots&t^{N}d\mu_0(t)\\
      t^2d\mu_0(t) &t^3d\mu_0(t)&\dots&t^{N+1}d\mu_0(t)\\
      \vdots& \vdots& &\vdots \\
      t^{N-1}d\mu_0(t) &t^Nd\mu_0(t)&\dots&t^{2N-2}d\mu_0(t) \\
   \end{pmatrix}
\]
and $p_n$ are monic polynomials orthogonal with respect to the measure $d\mu_0$.
At the same time, the matrix $J_{mon}^*$ corresponds to Sobolev type  orthogonal polynomials which, in turn, yield a sequence of matrix
orthogonal polynomials with respect to the measure \cite{DurVanAss}
\begin{equation}\label{MgerTr}
dM(h^{-1}(t))+L\delta(t),
\end{equation}
where $\delta (t)$ is the Dirac delta at $t=0$, $dM$ has the form
\[
dM(t)=\begin{pmatrix}
      d\mu (t) &td\mu(t)&\dots&t^{N-1}d\mu(t) \\
      td\mu (t) &t^2d\mu(t)&\dots&t^{N}d\mu(t)\\
      t^2d\mu (t) &t^3d\mu(t)&\dots&t^{N+1}d\mu(t)\\
      \vdots& \vdots& &\vdots \\
      t^{N-1}d\mu (t) &t^Nd\mu(t)&\dots&t^{2N-2}d\mu(t) \\
   \end{pmatrix},
\]
and $L$ is the matrix
\[
\sum_{i=1}^{M}\sum_{j=0}^{M_i}\lambda_{i,j}{\bf L}(i,j)
\]
with ${\bf L}(i,j)$ the $N\times N$ matrix
\[
{\bf L}(i,j)= \begin{pmatrix}
      0 \\
      \vdots \\
      0 \\
       j! \\
      \vdots \\
         \frac{k!}{(k-j)!}c_i^{k-j} \\
      \vdots \\
       \frac{(N-1)!}{(N-1-j)!}c_i^{N-1-j}\\
   \end{pmatrix}
    \begin{pmatrix}
      0 \dots 0&j!\dots \frac{k!}{(k-j)!}c_i^{k-j}\dots  \frac{(N-1)!}{(N-1-j)!}c_i^{N-1-j}\\
   \end{pmatrix}.
\]

In other words, we see that, according to \eqref{MatrGer}, the matrix measure \eqref{MgerTr} is actually a simple matrix Geronimus transformation of the matrix measure $dM_0$. In fact, introducing $y=h^{-1} (t)$ we see that the spectral transformation
\[
dM_0(y)=ydM(y)\mapsto dM(y)+L\delta (y)
\]
corresponds to one step of the block $LR$-algorithm based on the block $UL$-factorization
\[
h(J_{mon})=U_{mon}L_{mon}\mapsto J_{mon}^*=L_{mon}U_{mon}.
\]
Thus, a multiple Geronimus transformation is a simple Geronimus transformation for matrix inner products. So,
all our findings can be considered from the point of view of Darboux transformations for matrix orthogonal polynomials, which will be carefully analyzed in a forthcoming paper.

\vspace{10 mm}

\noindent{\bf Acknowledgements.}
Maxim Derevyagin acknowledges the support of FWO Flanders project G.0934.13 and
Belgian Interuniversity Attraction Pole P07/18. The work of Francisco Marcell\'an has been supported by Direcci\'on General de Investigaci\'on, Desarrollo e Innovaci\'on, Ministerio de Econom\'ia y Competitividad  of Spain, grant MTM2012-36732-C03-01. 


\renewcommand{\refname}{Bibliography}

\end{document}